\documentclass[12pt, reqno]{amsart}

\usepackage{amsmath, amsthm, amscd, amsfonts, amssymb, graphicx, color}
\usepackage[bookmarksnumbered, colorlinks, plainpages]{hyperref}

\newtheorem{theorem}{Theorem}[section]
\newtheorem{lemma}[theorem]{Lemma}

\newtheorem{corollary}[theorem]{Corollary}
\theoremstyle{definition}

\theoremstyle{remark}
\newtheorem{remark}[theorem]{Remark}

\begin{document}
\setcounter{page}{1}

\title[Some Inequalities for the Ratios of Generalized Digamma Functions]{Some Inequalities for the Ratios of Generalized Digamma Functions}

\author[Kwara Nantomah]{Kwara Nantomah$^{*}$$^1$}

\address{$^{1}$ Department of Mathematics, University for Development Studies, Navrongo Campus, P. O. Box 24, Navrongo, UE/R, Ghana. }
\email{\textcolor[rgb]{0.00,0.00,0.84}{mykwarasoft@yahoo.com, knantomah@uds.edu.gh}}




\begin{abstract}
Some inequalities for the ratios of generalized digamma functions are presented. The approache makes use of the series representations of the $(q,k)$-digamma and $(p,q)$-digamma functions.
\end{abstract} \maketitle

\section{Introduction and Preliminaries}
\noindent
The classical Euler's Gamma function $\Gamma(t)$ and the digamma function $\psi(t)$  are commonly  defined as
\begin{equation*}\label{eqn:gamma-digamma}
\Gamma(t)=\int_0^\infty e^{-x}x^{t-1}\,dx, \qquad \psi(t)=\frac{d}{dt}\ln \Gamma(t) =\frac{\Gamma'(t)}{\Gamma(t)},                                     \quad t>0
\end{equation*}

\noindent
In 2005, D\'{i}az and Teruel  \cite{Diaz-Teruel-2005} defined the  $(q,k)$-Gamma function, $\Gamma_{q,k}(t)$ as 
\begin{equation*}\label{eqn:(q,k)-gamma}
\Gamma_{q,k}(t)=\frac{  (1-q^{k})_{q,k}^{\frac{t}{k}-1}}  {(1-q)^{\frac{t}{k}-1}}=
\frac{ (1-q^{k})_{q,k}^{\infty}}  { (1-q^{t})_{q,k}^{\infty}(1-q)^{\frac{t}{k}-1}} ,\quad t>0, \, k>0, \, q\in(0,1).
\end{equation*}

\noindent
with the  $(q,k)$-digamma function, $\psi_{q,k}(t)$  is defined as
\begin{equation*}\label{eqn:(q,k)-digamma}
\psi_{q,k}(t)=\frac{d}{dt}\ln \Gamma_{q,k}(t) =\frac{\Gamma'_{q,k}(t)}{\Gamma_{q,k}(t)}, \quad t>0, \, k>0, \, q\in(0,1).
\end{equation*}

\noindent
Also in 2012, Krasniqi and Merovci \cite{Krasniqu-Merovci-2012} gave the $(p,q)$-Gamma function, $\Gamma_{p,q}(t)$ as
\begin{equation*}\label{eqn:(p,q)-gamma}
\Gamma_{p,q}(t)=\frac{[p]_{q}^{t}[p]_{q}!}{[t]_{q}[t+1]_{q}\dots[t+p]_{q}} ,\quad t>0, \, p\in N, \, q\in(0,1).
\end{equation*}
where \, $[p]_{q}=\frac{1-q^p}{1-q}$.\\

\noindent
Similarly, the $(p,q)$-digamma function, $\psi_{p,q}(t)$  is defined as
\begin{equation*}\label{eqn:(p,q)-digamma}
\psi_{p,q}(t)=\frac{d}{dt}\ln \Gamma_{p,q}(t) =\frac{\Gamma'_{p,q}(t)}{\Gamma_{p,q}(t)}, \quad t>0, \, p\in N, \, q\in(0,1).
\end{equation*}

\noindent
The functions $\psi_{q,k}(t)$ and  $\psi_{p,q}(t)$  as defined above exhibit the following series representations.
\begin{align}
\psi_{q,k}(t)&=\frac{-\ln(1-q)}{k} + (\ln q)\sum_{n=1}^{\infty}\frac{q^{nkt}}{1-q^{nk}}, \qquad t>0. \label{eqn:(q,k)-series-digamma}\\
\psi_{p,q}(t)&= \ln[p]_q  + (\ln q)\sum_{n=1}^{p}\frac{q^{nt}}{1-q^{n}}, \qquad t>0. \label{eqn:(p,q)-series-digamma}
\end{align}

\noindent
By taking derivatives of these functions, it can easily be established that, 

\begin{align}
\psi_{q,k}'(t)&= (\ln q)^{2}\sum_{n=1}^{\infty}\frac{nk.q^{nkt}}{1-q^{nk}},\quad t>0 \label{eqn:(q,k)-series-digamma-prime} \\
\psi_{p,q}'(t)&= (\ln q)^{2}\sum_{n=1}^{p}\frac{n.q^{nt}}{1-q^{n}},\quad t>0. \label{eqn:(p,q)-series-digamma-prime} 
\end{align}

\noindent
In \cite{Nantomah-2014}, Nantomah presented the following results for the digamma function. 
\begin{equation}\label{eqn:digamma-ineq}
\frac{\left[ \psi(a)\right] ^{\alpha}}{\left[ \psi(c)\right] ^{\beta}} \leq
\frac{\left[ \psi(a+bt)\right] ^{\alpha}}{\left[ \psi(c+dt)\right] ^{\beta}} \leq
\frac{\left[ \psi(a+b)\right] ^{\alpha}}{\left[ \psi(c+d)\right] ^{\beta}}
\end{equation}
where  $a$, $b$,  $c$, $d$, $\alpha$, $\beta$ are  positive real numbers such that  $\beta d \leq \alpha b$, $ a+bt \leq c+dt$, $\psi(a+bt)>0$ and $\psi(c+dt)>0$. \\

\noindent
Also, the  $k$-analogue of these inequalities can  be found in \cite{Nantomah-Iddrisu-2014}.\\

\noindent
The purpose of this paper is to extend  inequalities ~(\ref{eqn:digamma-ineq}) to the $(q,k)$ and $ (p,q)$-digamma functions.


\section{Results and Discussion}
\noindent
We now present the results of this paper.

\begin{lemma}\label{lem:(q,k)-digamma-increasing}
Let  $0<s\leq t$, then the following statement is valid.
\begin{equation*}\label{eqn:(q,k)-digamma-increasing}
\psi_{q,k}(s) \leq \psi_{q,k}(t).   
\end{equation*}
\end{lemma}

\begin{proof}
From ~(\ref{eqn:(q,k)-series-digamma}) we have,
\begin{equation*}
\psi_{q,k}(s) - \psi_{q,k}(t) = (\ln q)\sum_{n=1}^{\infty} \left[ \frac{q^{nks}-q^{nkt}}{1-q^{nk}} \right]\leq0.
\end{equation*}
\end{proof}

\begin{lemma}\label{lem:(q,k)-digamma-prime-decreasing}
Let  $0<s\leq t$, then the following statement is valid.
\begin{equation*}\label{eqn:(q,k)-digamma-prime-decreasing}
\psi_{q,k}'(s) \geq \psi_{q,k}'(t).   
\end{equation*}
\end{lemma}

\begin{proof}
From ~(\ref{eqn:(q,k)-series-digamma-prime}) we have,
\begin{equation*}
\psi_{q,k}'(s) - \psi_{q,k}'(t) = (\ln q)^{2}\sum_{n=1}^{\infty} \left[ \frac{nk(q^{nks}-q^{nkt})}{1-q^{nk}} \right]\geq0.
\end{equation*}
\end{proof}


\begin{lemma}\label{lem:(q,k)-digammas}
Let  $a$, $b$,  $c$, $d$, $\alpha$, $\beta$ be  positive real numbers such that  $a+bt \leq c+dt$, $\beta d \leq \alpha b$,  $\psi_{q,k}(a+bt)>0$ and $\psi_{q,k}(c+dt)>0$. Then
\begin{equation*}
\alpha b \psi_{q,k}(c+dt)\psi'_{q,k}(a+bt) - \beta d \psi_{q,k}(a+bt)\psi'_{q,k}(c+dt)\geq0.    
\end{equation*}
\end{lemma}

\begin{proof}
Since $0<a+bt \leq c+dt$, then by Lemmas \ref{lem:(q,k)-digamma-increasing} and \ref{lem:(q,k)-digamma-prime-decreasing} we have,\\
$0<\psi_{q,k}(a+bt)\leq \psi_{q,k}(c+dt)$ and $\psi'_{q,k}(a+bt)\geq \psi'_{q,k}(c+dt)>0$. \\ Then that implies;\\
$\psi_{q,k}(c+dt)\psi'_{q,k}(a+bt) \geq \psi_{q,k}(c+dt)\psi'_{q,k}(c+dt) \geq \psi_{q,k}(a+bt)\psi'_{q,k}(c+dt)$.\\ Further, $\alpha b \geq \beta d$ implies;\\
$\alpha b \psi_{q,k}(c+dt)\psi'_{q,k}(a+bt) \geq \alpha b \psi_{q,k}(a+bt)\psi'_{q,k}(c+dt)  \geq \beta d \psi_{q,k}(a+bt)\psi'_{q,k}(c+dt)$. \\ Hence,\\
$\alpha b \psi_{q,k}(c+dt)\psi'_{q,k}(a+bt) - \beta d \psi_{q,k}(a+bt)\psi'_{q,k}(c+dt)\geq0$. 
\end{proof}

\begin{theorem}\label{thm:(q,k)-digamma-funct}
Define a function $G$ by 
\begin{equation}\label{eqn:(q,k)-digamma-funct} 
G(t)=\frac{\left[ \psi_{q,k}(a+bt)\right] ^{\alpha}}{\left[ \psi_{q,k}(c+dt)\right] ^{\beta}},  \quad t\in [0,\infty)
\end{equation}
where  $a$, $b$,  $c$, $d$, $\alpha$, $\beta$ are  positive real numbers such that  $a+bt \leq c+dt$, $\beta d \leq \alpha b$,  $\psi_{q,k}(a+bt)>0$ and $\psi_{q,k}(c+dt)>0$. Then $G$ is nondecreasing on $t\in[0,\infty)$ and the inequalities
\begin{equation}\label{eqn:(q,k)-digamma-funct-ineq}
\frac{\left[ \psi_{q,k}(a)\right] ^{\alpha}}{\left[ \psi_{q,k}(c)\right] ^{\beta}} \leq
\frac{\left[ \psi_{q,k}(a+bt)\right] ^{\alpha}}{\left[ \psi_{q,k}(c+dt)\right] ^{\beta}} \leq
\frac{\left[ \psi_{q,k}(a+b)\right] ^{\alpha}}{\left[ \psi_{q,k}(c+d)\right] ^{\beta}}
\end{equation}
are valid for every $t\in[0,1]$.
\end{theorem}

\begin{proof}
Let $g(t)=\ln G(t)$ for every $t\in[0,\infty)$. Then,
\begin{align*}
g &=\ln \frac{\left[ \psi_{q,k}(a+bt)\right] ^{\alpha}}{\left[ \psi_{q,k}(c+dt)\right] ^{\beta}}
        = \alpha \ln \psi_{q,k}(a+bt) - \beta \ln \psi_{q,k}(c+dt)   \\
\intertext{and}
g'(t)&= \alpha b \frac{\psi'_{q,k}(a+bt)}{\psi_{q,k}(a+bt)} - \beta d \frac{\psi'_{q,k}(c+dt)}{\psi_{q,k}(c+dt)} \\
      &= \frac{\alpha b \psi'_{q,k}(a+bt)\psi_{q,k}(c+dt) - \beta d \psi'_{q,k}(c+dt)\psi_{q,k}(a+bt) }{\psi_{q,k}(a+bt) \psi_{q,k}(c+dt) }\geq0 \quad
\end{align*}
as a result of Lemma \ref{lem:(q,k)-digammas}.  
That implies $g$ as well as  $G$ are nondecreasing on $t\in[0,\infty)$ and  for every $t\in[0,1]$ we have, 
\begin{equation*}
G(0) \leq G(t) \leq G(1)   
\end{equation*}
concluding the proof.
\end{proof}

\begin{corollary}
If $t\in(1,\infty)$, then the following inequality is valid.
\end{corollary}
\begin{equation}
\frac{\left[ \psi_{q,k}(a+bt)\right] ^{\alpha}}{\left[ \psi_{q,k}(c+dt)\right] ^{\beta}} \geq
\frac{\left[ \psi_{q,k}(a+b)\right] ^{\alpha}}{\left[ \psi_{q,k}(c+d)\right] ^{\beta}}
\end{equation}

\begin{proof}
For each $t\in(1,\infty)$, we have \,$G(t)\geq G(1)$\, yielding the result.
\end{proof}

\begin{lemma}\label{lem:(p,q)-digamma-increasing}
Let  $0<s\leq t$, then the following statement is valid.
\begin{equation*}\label{eqn:(p,q)-digamma-increasing}
\psi_{p,q}(s) \leq \psi_{p,q}(t).   
\end{equation*}
\end{lemma}

\begin{proof}
From ~(\ref{eqn:(p,q)-series-digamma}) we have,
\begin{equation*}
\psi_{p,q}(s) - \psi_{p,q}(t) = (\ln q)\sum_{n=1}^{p} \left[ \frac{q^{ns}-q^{nt}}{1-q^{n}} \right]\leq0.
\end{equation*}
\end{proof}

\begin{lemma}\label{lem:(p,q)-digamma-prime-decreasing}
Let  $0<s\leq t$, then the following statement is valid.
\begin{equation*}\label{eqn:(p,q)-digamma-prime-decreasing}
\psi_{p,q}'(s) \geq \psi_{p,q}'(t).   
\end{equation*}
\end{lemma}

\begin{proof}
From ~(\ref{eqn:(p,q)-series-digamma-prime}) we have,
\begin{equation*}
\psi_{p,q}'(s) - \psi_{p,q}'(t) = (\ln q)^{2}\sum_{n=1}^{p} \left[ \frac{n(q^{ns}-q^{nt})}{1-q^{n}} \right]\geq0.
\end{equation*}
\end{proof}


\begin{lemma}\label{lem:(p,q)-digammas}
Let  $a$, $b$,  $c$, $d$, $\alpha$, $\beta$ be  positive real numbers such that  $a+bt \leq c+dt$, $\beta d \leq \alpha b$,  $\psi_{p,q}(a+bt)>0$ and $\psi_{p,q}(c+dt)>0$. Then
\begin{equation*}
\alpha b \psi_{p,q}(c+dt)\psi'_{p,q}(a+bt) - \beta d \psi_{p,q}(a+bt)\psi'_{p,q}(c+dt)\geq0.    
\end{equation*}
\end{lemma}

\begin{proof}
Follows the same argument as in the proof of Lemma \ref{lem:(q,k)-digammas}.
\end{proof}

\begin{theorem}\label{thm:(p,q)-digamma-funct}
Define a function $H$ by 
\begin{equation}\label{eqn:(p,q)-digamma-funct} 
H(t)=\frac{\left[ \psi_{p,q}(a+bt)\right] ^{\alpha}}{\left[ \psi_{p,q}(c+dt)\right] ^{\beta}},  \quad t\in [0,\infty)
\end{equation}
where  $a$, $b$,  $c$, $d$, $\alpha$, $\beta$ are  positive real numbers such that  $a+bt \leq c+dt$, $\beta d \leq \alpha b$,  $\psi_{p,q}(a+bt)>0$ and $\psi_{p,q}(c+dt)>0$. Then $H$ is nondecreasing on $t\in[0,\infty)$ and the inequalities
\begin{equation}\label{eqn:(p,q)-digamma-funct-ineq}
\frac{\left[ \psi_{p,q}(a)\right] ^{\alpha}}{\left[ \psi_{p,q}(c)\right] ^{\beta}} \leq
\frac{\left[ \psi_{p,q}(a+bt)\right] ^{\alpha}}{\left[ \psi_{p,q}(c+dt)\right] ^{\beta}} \leq
\frac{\left[ \psi_{p,q}(a+b)\right] ^{\alpha}}{\left[ \psi_{p,q}(c+d)\right] ^{\beta}}
\end{equation}
are valid for every $t\in[0,1]$.
\end{theorem}

\begin{proof}
Follows the same procedure as in Theorem \ref{thm:(q,k)-digamma-funct}. Using Lemma \ref{lem:(p,q)-digammas}, we conclude that  $H$ is nondecreasing on $t\in[0,\infty)$ and  for every $t\in[0,1]$ we have, 
\begin{equation*}
H(0) \leq H(t) \leq H(1)   
\end{equation*}
ending the proof.
\end{proof}

\begin{corollary}
If $t\in(1,\infty)$, then the following inequality is valid.
\end{corollary}
\begin{equation}
\frac{\left[ \psi_{p,q}(a+bt)\right] ^{\alpha}}{\left[ \psi_{p,q}(c+dt)\right] ^{\beta}} \geq
\frac{\left[ \psi_{p,q}(a+b)\right] ^{\alpha}}{\left[ \psi_{p,q}(c+d)\right] ^{\beta}}
\end{equation}

\begin{proof}
For each $t\in(1,\infty)$, we have \,$H(t)\geq H(1)$\, yielding the result.
\end{proof}

\section{Concluding Remarks}
We dedicate this section to some remarks concerning our results.
\begin{remark}
If in  ~(\ref{eqn:(q,k)-digamma-funct-ineq})  we allow  $k\rightarrow 1$,  then we obtain the $q$-analogue of  ~(\ref{eqn:digamma-ineq}).
\end{remark}

\begin{remark}
If in  ~(\ref{eqn:(q,k)-digamma-funct-ineq})  we allow  $q\rightarrow 1^-$,   then we obtain the $k$-analogue of  ~(\ref{eqn:digamma-ineq}) as presented  in Theorem 3.7 of the paper \cite{Nantomah-Iddrisu-2014}.
\end{remark}

\begin{remark}
If in  ~(\ref{eqn:(q,k)-digamma-funct-ineq})  we allow  $q\rightarrow 1^-$ as $k\rightarrow 1$,   then we obtain ~(\ref{eqn:digamma-ineq}).
\end{remark}

\begin{remark}
If  in ~(\ref{eqn:(p,q)-digamma-funct-ineq}) we allow  $q\rightarrow 1^-$, then we obtain the $p$-analogue of  ~(\ref{eqn:digamma-ineq}).
\end{remark}

\begin{remark}
If  in ~(\ref{eqn:(p,q)-digamma-funct-ineq}) we allow $p\rightarrow \infty$, then we obtain the $q$-analogue of  ~(\ref{eqn:digamma-ineq}).
\end{remark}

\begin{remark}
If  in ~(\ref{eqn:(p,q)-digamma-funct-ineq}) we allow  $p\rightarrow \infty$ as $q\rightarrow 1^-$, then we obtain ~(\ref{eqn:digamma-ineq}).\\
\end{remark}

\noindent
{\bf Conflict of Interests.} The authors declare that there is no conflict of interests.

\bibliographystyle{plain}


\end{document}